\documentclass[a4paper,12pt]{article}
\usepackage{amsthm}
\usepackage{amsfonts}
\usepackage{amsmath}
\usepackage{bm}
\usepackage{amscd}
\usepackage{comment}
\usepackage{epsfig}
\usepackage[none]{hyphenat}

\newtheorem{theorem}{Theorem}

\newtheorem{lemma}[theorem]{Lemma}
\newtheorem{prop}[theorem]{Proposition}

\theoremstyle{definition}

\theoremstyle{plain}

\def\xideg#1{\partial_{#1}}
\def\degree{\partial}

\title{Euclid meets B\'ezout: Intersecting  algebraic plane curves with the Euclidean
algorithm}
\author{Jan Hilmar and Chris Smyth}
\date{}
\begin{document}
\maketitle

\begin{section}{Introduction}

We can be quite sure that Euclid ($\sim\!\! 325$  to $\sim\!\! 265$ BC) and \'Etienne B\'ezout (1730--83) never met. But we show here how Euclid's algorithm for polynomials can be used to find, with their  multiplicities, the points of intersection of two algebraic plane curves.  As a consequence, we obtain a simple proof of B\'ezout's Theorem, giving the total number of such intersections. 

We'd perhaps expect two such plane curves to be given by equations like $\sum_{i,j}a_{ij}x^iy^j=0$ and $\sum_{i,j}b_{ij}x^iy^j=0$, with coefficients in some field $K$, and ask for the points $(x,y)$ in $K^2$ lying on both curves.
However, this question has a nicer answer if it is tweaked a bit, so
 we modify the question in several ways. First of all, we seek points with coordinates in $\overline K$, the algebraic closure of $K$, instead of just in $K$. Secondly, we work with homogeneous polynomials $A(x,y,z)=\sum_{i,j}a_{ij}x^iy^jz^{m-i-j}$, where every term $a_{ij}x^iy^jz^{m-i-j}$ has the same  degree $i+j+(m-i-j)=m$, the  degree of $A$. Note that the point $(0,0,0)$ always lies on $A=0$, and that for every point $(x,y,z)$ on $A=0$ and every $\lambda$ the point $(\lambda x,\lambda y,\lambda z)$ also lies on $A=0$. Thus we would like to ignore $(0,0,0)$ and also regard 
$(x,y,z)$ and $(\lambda x,\lambda y,\lambda z)$ when $\lambda\ne 0$ as essentially the same point.
This brings us to our third tweak: we say that two nonzero points in $\overline K^3$ are {\it equivalent} if each is a  scalar multiple of the other. The equivalence classes of the resulting equivalence relation give us the projective plane $\overline K\mathbb P^2$. Then, choosing equivalence class representatives, we can for our purposes regard $\overline K\mathbb P^2$ as consisting of the points in $\overline K^3$ of the form $(x,y,1)$, $(x,1,0)$ and $(1,0,0)$.

Finally, we count our intersection points with {\it multiplicity}: just as the parabola $y=x^2$ intersects the $y$-axis with multiplicity $1$ but the (tangential) $x$-axis with multiplicity $2$, we attach a suitable positive integer as the multiplicity of every intersection point. Then take $A(x,y,z)$ and another homogeneous polynomial $B(x,y,z)=\sum_{i,j}b_{ij}x^iy^jz^{n-i-j}$, and ask our modified question:

How many intersection points are there of $A(x,y,z)=0$ and $B(x,y,z)=0$ in $\overline K\mathbb P^2$, counted with multiplicity, and how do we find them?

 The number of points is given by B\'ezout's Theorem:
 
 \begin{theorem}[B{\'e}zout's Theorem]\label{Th-B}
Let $A,B\in K[x,y,z]$  be homogeneous of degrees $m,n$ respectively, with
no nonconstant common factor.  Then  in $\overline K\mathbb P^2$ the curves $A=0$ and $B=0$
intersect in exactly $mn$ points, counting multiplicities.
\end{theorem}

We give a simple proof of this result in Section \ref{S-B}. The algorithm given in Section \ref{algorithm} calculates these points, and their multiplicities. 

B{\'e}zout's Theorem also gives us an answer to our original (untweaked) question: we get rid of $z$ by setting it to $1$, and then the number of intersection points of  $\sum_{i,j}a_{ij}x^iy^j=0$ and $\sum_{i,j}b_{ij}x^iy^j=0$ is the number of points of the form $(x,y,1)$ with $x,y$ in $K$ lying on both homogeneous curves. Thus there are at most $mn$ of them.

 B{\'e}zout's Theorem  is a generalization of the Fundamental Theorem of Algebra, telling us
 that a polynomial $f(x)$ of degree $n$ with complex coefficients has $n$
 complex roots.    (The curves $y=f(x)$ and $y=0$ are replaced by arbitrary ones, and in projective space.)

The special case $m=n=1$ of B{\'e}zout's Theorem  tells us that two (distinct) lines in the projective plane always intersect at a point (no parallel lines in $\overline K\mathbb P^2$!). But in general finding the intersection points, and especially their multiplicities, is a nontrivial business. It is this process which we aim to demystify here, by reducing the general case to the case $m=n=1$.

The intersection of two curves $A=0$ and $B=0$ can be expressed as a formal sum $A\cdot B$ of their intersection points, called the {\it intersection cycle}, defined below. The idea of the algorithm is to use  the steps of the Euclidean algorithm
 to express  $A\cdot B$ in terms of intersection
 cycles of curves defined by polynomials of lower and lower $x$-degree. 
 In the end, we can write $A\cdot B$
  in terms of intersection cycles of $2$-variable homogeneous polynomials. But these are simply products of lines, whose intersection
  points can be written down immediately (see Proposition \ref{a.bprops}(d) below).

\end{section}

\begin{section}{Intersection Cycles of Algebraic Curves}\label{S-three}

Let $K$ be a field and denote by $\overline K{\mathbb P}^2$  the projective plane
over $\overline K$. For a homogeneous
polynomial $A(x,y,z)\in K[x,y,z]$, we will abuse notation slightly
 by identifying it with the curve $A=0$ in $\overline K\mathbb P^2$.
Further, let
$\xideg{x}A$ denote the $x$-degree of
the polynomial $A(x,y,z)$ and  $\degree A$ its (total) degree. While the $\gcd$ of $A$ and $B$ is defined only up to multiplication by a scalar, we write $\gcd(A,B)=1$ for two such
curves $A$ and $B$ if they have no nonconstant  common
factor. Clearly  $A$ and any nonzero scalar multiple $\lambda A$ of $A$ define the same curve.
From now on, all polynomials in upper case ($A$, $B$, $C$, \dots)  will be assumed to be homogeneous.

 For any point $\mathbf{P}$ in
 $\overline K\mathbb P^2$, and curves $A$ and $B$, we denote by $i_\mathbf{P}(A,B)$ the intersection
multiplicity of the curves $A$ and $B$ at $\mathbf{P}$. This is a nonnegative integer, positive if $\mathbf{P}$ lies on both $A$ and $B$, and otherwise zero. We seek the formal sum $A\cdot
B=\sum_\mathbf{P} i_\mathbf{P}(A,B)\mathbf{P}$,  the  {\it intersection cycle} of $A$
and $B$, which is simply an object for recording the intersection of
these curves. Our algorithm does not need to use the definition of $i_\mathbf{P}(A,B)$ (for this, see the appendix), only the standard properties of intersection cycles in the following proposition.
\begin{prop}\label{a.bprops}
Let $A,B$ and $C$ be algebraic curves with $\gcd(A,B)=\gcd(A,C)=1$. Then
\begin{itemize}

\item[$(a)$] $A\cdot B=B\cdot A$;
\item[$(b)$] $A\cdot(BC)=A\cdot B+A\cdot C$;
\item[$(c)$] $A\cdot(B+AC)=A\cdot B$ if $\degree B=\degree (AC)$;
\item[(d)] If $A$ and $B$ are distinct lines, say  $A(x,y,z)=a_1x+a_2y+a_3z$ and  $B(x,y,z)=b_1x+b_2y+b_3z$, then their intersection cycle $A\cdot B$ is the single point $\mathbf{P}_\times$ given by
\begin{eqnarray}
\mathbf{P}_\times&=&\left(
\left|\begin{array}{cc}
a_2&a_3\\
b_2&b_3
\end{array}\right|,
\left|\begin{array}{cc}
a_3&a_1\\
b_3&b_1
\end{array}\right|,
\left|\begin{array}{cc}
a_1&a_2\\
b_1&b_2
\end{array}\right|
\right).\label{l1.l2}
\end{eqnarray}
\end{itemize}
\end{prop}

These properties are quite natural: part (a) just says that the intersection points don't depend on the order of the curves, while part (b) tells us that the points on $A$ and $BC$ are the points on $A$ and $B$ plus the points on $A$ and $C$, and that the multiplicities add. For part (c), we clearly need the condition $\degree B=\degree (AC)$ to make $B+AC$ homogeneous. Then any point on $A$ and $B$ will also lie on $B+AC$. The fact that the multiplicity at each intersection point is the same comes from the fact (see appendix) that the multiplicity is defined in terms of an ideal generated by the two curves, and $A$ and $B$ generate the same ideal as $A$ and $B+AC$.

The proof of this Proposition follows straight from Lemma \ref{L-i} in the appendix, where we state and prove corresponding properties of the intersection
multiplicity $i_\mathbf{P}(A,B)$.

\end{section}

\begin{section}{The Algorithm}\label{algorithm}

\begin{subsection}{The Euclidean part}\label{S-3.1}

Let $A,B\in K[x,y,z]$ be algebraic curves with $\gcd(A,B)=1$ and, say,
$\xideg{x}A\geq\xideg{x}B\ge 1$. By polynomial division we can find $q,
r\in K(y,z)[x]$ with
\begin{eqnarray*}
A&=&qB+r
\end{eqnarray*}
and $0\leq\xideg{x}r<\xideg{x}B$ and $q, r\ne 0$. Since the coefficients of $q$ and $r$ are rational functions of $y$ and $z$, we can multiply through by the
least common multiple $H\in K[y,z]$ of their denominators to get
\begin{eqnarray*}
HA&=&QB+R,
\end{eqnarray*}
where $Q=qH,R=rH\in K[x,y,z]$ are both homogeneous. Since $HA$ is homogeneous, too, $\degree(QB)=\degree R$. Suppose now that $G=\gcd(B,R)$. As $\gcd(A,B)=1$, it is clear that also $\gcd(B,H)=G$, so we can divide
through by $G$ to get
\begin{eqnarray}
H'A&=&QB'+R',\label{Euclid}
\end{eqnarray}
where $B=B'G$, $H=H'G$, $R=R'G$, and  $\gcd(B',R')=\gcd(B',H')=1$. 
Now
\begin{align}
A\cdot B &= A\cdot (B'G)&\notag  \\
&=A\cdot B' + A\cdot G \qquad &\text{(by Proposition \ref{a.bprops}(b))}\notag \\
&=(H'A)\cdot B' -H'\cdot B'+A\cdot G &\text{(by Proposition \ref{a.bprops}(b) again)}\notag  \\
&=(QB'+R')\cdot B' -H'\cdot B'+A\cdot G\qquad&\text{(using (\ref{Euclid}))}\notag  \\
 &=R'\cdot B'-H'\cdot B'+A\cdot G \qquad&\text{(by Proposition \ref{a.bprops}(c)).}\label{E-3}
\end{align}

Note that as $G$ and $H'$ are both  factors of $H\in K[y,z]$, we have $\xideg{x}G=\xideg{x}H'=0$ and $\xideg{x}B'= \xideg{x}B$. Also, because 
$$
\xideg{x}R'\le \xideg{x}r<\xideg{x}B\le \xideg{x}A,
$$
we see that the first intersection cycle $R'\cdot B'$ on the right-hand side of
(\ref{E-3}) has the property that the minimum of the $x$-degrees of its curves is less than the minimum of the $x$-degrees of the curves of $A\cdot B$, while the second and third intersection cycles both have one curve with $x$-degree $0$. Thus by next applying (\ref{E-3})  to $R'\cdot B'$, and proceeding recursively, we can express $A\cdot B$ as a sum of terms $\pm C\cdot D$, where $C\in K[x,y,z]$ and $D\in K[y,z]$. We have thus reduced the problem of computing $A\cdot B$ to computing such simpler intersection cycles.

\end{subsection}

\begin{subsection}{Intersecting a curve with a product of
lines}\label{S-lines}
 Given $C\in K[x,y,z]$ and $D\in K[y,z]$, we
first note that, because of Proposition \ref{a.bprops}(b), we can
assume that $D$ is irreducible over $K$. If $D$ doesn't contain the variable $y$, then, being irreducible, it must be $z$. Otherwise, over $\overline K$ it will factor as, say, 
\begin{equation}\label{E-d}
D(y,z)=\prod_\beta(y-\beta z),
\end{equation}
 where the $\beta$ are the roots in $\overline K$ of $D(y,1)$. Thus $D$ is a product of lines.
Then since
\begin{align*}
C(x,y,z) &= C(x,y,0)+zC'(x,y,z)\\
\intertext{and also}
 C(x,y,z) &=C(x,\beta z,z)+(y-\beta z)C''(x,y,z)
\end{align*}
for some $C',C''$ in $K[x,y,z]$ we have by Proposition \ref{a.bprops}(c) that
\begin{align}
C\cdot z &=C(x,y,0)\cdot z\notag\\
\intertext{and}
 C\cdot (y-\beta z) &= C(x,\beta z,z)\cdot (y-\beta z).\label{E-6}
\end{align}
 Thus, either $D=z$ and $C\cdot
D=C(x,y,0)\cdot z$, or, using (\ref{E-d}), we have
\begin{eqnarray*}
C\cdot D&=&C(x,y,z)\cdot \left(\prod_{\beta}(y-\beta z)\right)\\
&=&\sum_\beta C(x,\beta z,z)\cdot (y-\beta z)\qquad\text{(by (\ref{E-6})).} 
\end{eqnarray*}

Next, in the case $D=z$, by factorizing $C(x,y,0)$ first into
irreducible factors over $K$, and then over its algebraic closure $\overline K$ (as either $y$
or a product $\prod_\alpha (x-\alpha y)$  of lines), we can reduce
the problem of finding $C\cdot D$ to  one of intersecting lines.
Specifically, for an irreducible factor $C_1(x,y)$ of $C(x,y,0)$ we get $C_1\cdot z=
(1,0,0)$ if $C_1=y$, and 
\begin{equation}\label{E-sum}
C_1\cdot z=\sum_\alpha (x-\alpha y)\cdot z=\sum_\alpha(\alpha,1,0)\qquad\text{(using (\ref{l1.l2}))}
\end{equation} 
 otherwise, where the $\alpha$ are the roots of
$C_1(x,1)$.

In the case $D(y,z)=\prod_\beta(y-\beta z)$, we first factorize
$C(x,\beta z,z)$  over $K(\beta)$. Taking $C_2(x,z)$ as a typical
factor, we have that either $C_2=z$ and 
$$
C_2\cdot D=\sum_\beta z\cdot (y-\beta z)=(\degree D)(1,0,0);
$$
or that, over $\overline K$, we have $C_2(x,z)=\prod_\gamma(x-\gamma z)$, where the $\gamma$ are the roots in $\overline K$ of $C_2(x,1)$, and 
$$
C_2\cdot D=\sum_\beta\sum_\gamma (x-\gamma z)\cdot (y-\beta z)=\sum_\beta\sum_\gamma(\gamma,\beta,1).
$$

\end{subsection}

\begin{subsection}{The result}

From our algorithm we see that the intersection cycle $A\cdot B$ is a sum or difference of
simpler sums of the following types:
\begin{itemize}
\item[(1)] The point $(1,0,0)$;
\item[(2)]  A sum $\sum_\alpha (\alpha,1,0)$, the sum being taken over roots $\alpha$ of
a monic polynomial $f\in K[x]$ irreducible over $K$; let us denote
this sum by $\mathcal C_0(f(x))$;
\item[(3)]  A double sum  $\sum_\beta\sum_\gamma(\gamma,\beta,1)$, where $\sum_\beta$ is taken over the roots
$\beta$ of some monic polynomial $g\in K[y]$ irreducible over $K$,
and where $\sum_\gamma$ is taken over the roots $\gamma$ of some
monic polynomial $h_\beta\in \ K(\beta)[x]$ irreducible over  $K(\beta)$. Then we can write $h_\beta$ as a $2$-variable polynomial $h(x,\beta)$ with coefficients in $K$, where the $\beta$-degree of $h$ is less than the degree of $g$;
denote our double sum by $\mathcal C_1(h(x,y),g(y))$. Thus  $h$ and $g$ will
 specify this intersection cycle canonically.
\end{itemize}

We note that $(1,0,0)$ and the sums in (2) and (3) are {\it Galois-invariant}: they are
 unchanged by the action of any automorphism of $\overline K$ that fixes $K$. Thus we call them {\it Galois cycles}.
 Any point $\mathbf{P}\in\overline K\mathbb P^2$ can appear in only one such cycle: the cycles do not overlap. Further,
 since $A\cdot B$ is  a formal sum of {\it positive} integer multiples of the intersection points of $A$ and $B$,
 any negative multiple of Galois cycles in the sum of sums the algorithm gives for $A\cdot B$ must be
 cancelled by positive multiples of the same cycles. Writing Galois
 cycles in a canonical way as in (1), (2), and (3) above enables us
 to actually carry out such cancellation by computer.
Thus, in the end, the algorithm will give $A\cdot B$ as a sum (no differences!) of Galois cycles.

{\bf Remarks.} 1. If $f$ is linear, then $\mathcal C_0(f(x))$ is a single point. Similarly, if $g$ and $h$ are linear, then $\mathcal C_1(h(x,y),g(y))$ is a single point.
 For example, $\mathcal C_0(x-2)=(2,1,0)$, while $\mathcal C_1(x-3,y-4)=(3,4,1)$.
More generally, $\mathcal C_0(f(x))$ is a formal sum of $\degree f$ points, while 
$\mathcal C_1(h(x,y),g(y))$ is a sum of $\xideg{x}h\,\degree g$ points.

 2. In the above analysis, we have in several places, in equation (\ref{E-sum}) for instance, summed over the roots of a polynomial irreducible over $K$. If the polynomial  has multiple roots (i.e., is inseparable), then of course for each factor $(x-\alpha y)^\ell$ we take $\ell$ copies of whatever is being summed. (This can in fact happen only over certain fields of finite characteristic $p$, in which case $\ell$ is a power of $p$. See \cite[Prop. 3.8, p. 530]{DF}.) 
 
 3. To obtain our expression for $A\cdot B$ as a sum of Galois cycles we needed to factorize some polynomials over $K$, and some over certain fields $K(\beta)$. For many fields there are algorithms for doing this, depending on the particular field; for instance, factorization over the field $K=\mathbb Q$ of rationals, and over finite extensions $\mathbb Q(\beta)$, is implemented in Maple. And only at the end, when we want to write the Galois cycles in the answer as sums of points, do we need to actually find the roots in $\overline K$ of these polynomials.
 
 4. In Sections \ref{S-3.1} and \ref{S-lines} we have brazenly taken for granted that certain polynomials ($Q$, $R$,\dots) are homogenous; so as not to interrupt the flow of the paper, we have left verification of these facts to the careful reader.

\end{subsection}
\begin{subsection}{Examples}
As an illustration of the method, we now look at two examples of using the Euclidean algorithm to compute the intersection cycle of two curves $A$ and $B$ defined over the rationals:

\begin{figure}[h]
\begin{center}
\leavevmode
\hbox{
\epsfxsize=3.0in
\epsffile{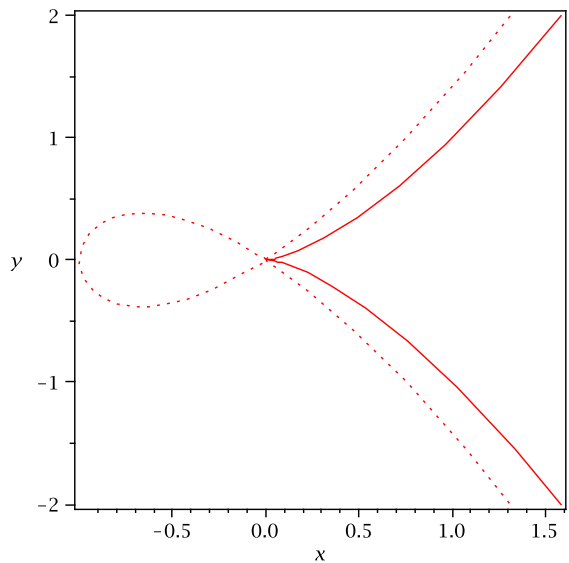}
}
\end{center}
\caption{The `slice' $z=1$ of the cubic curves $y^2z-x^3$ (solid line) and $y^2z-x^2(x+z)$ (dotted line) near $(0,0,1)$, an intersection point of multiplicity $4$. (These are the curves $y^2=x^3$ and $y^2=x^2(x+1)$.)} \label{F-cubics}
\end{figure}

\begin{figure}[h]
\begin{center}
\leavevmode
\hbox{
\epsfxsize=3.0in
\epsffile{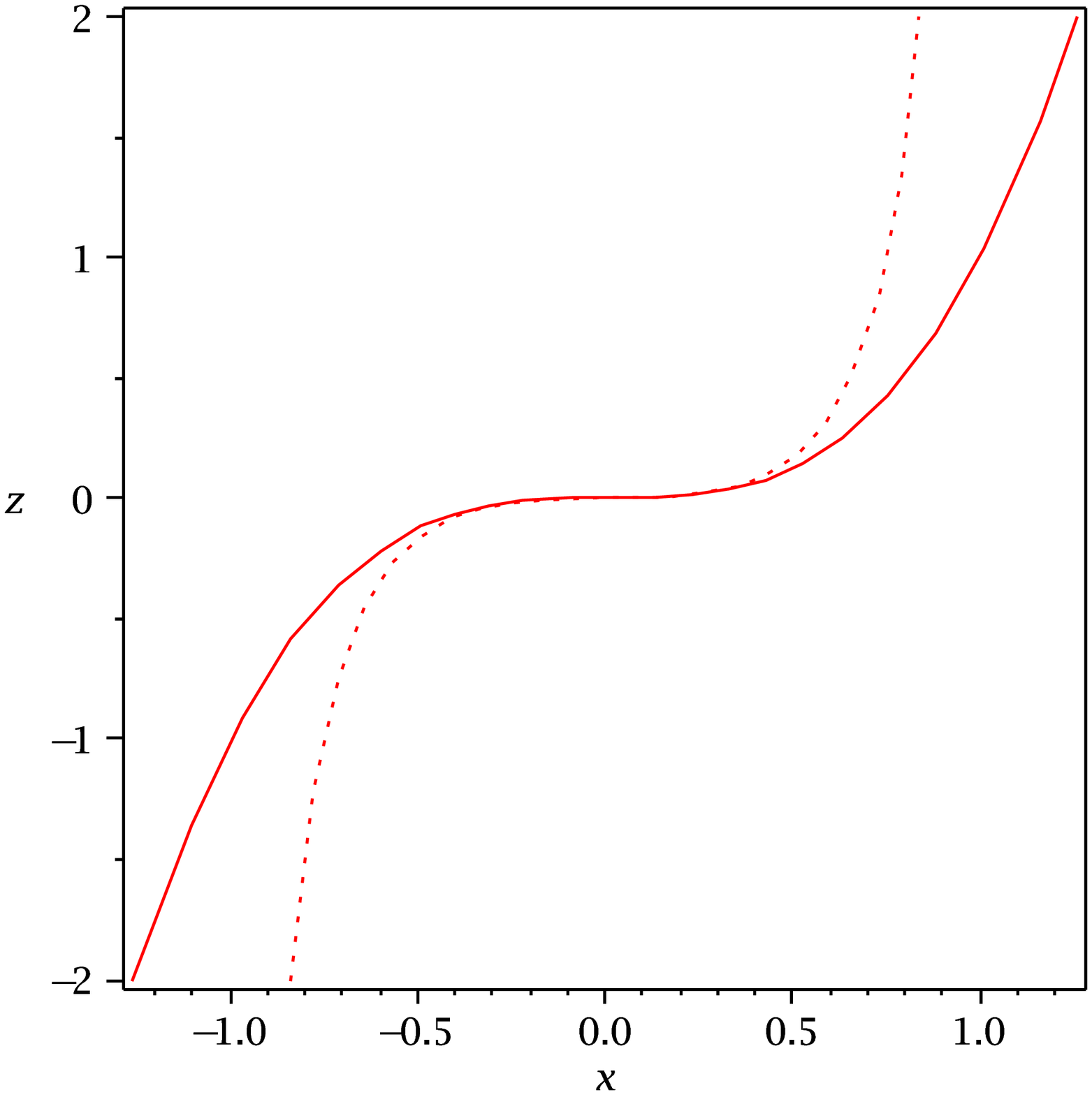}
}
\end{center}
\caption{The `slice' $y=1$ of the same curves $y^2z-x^3$ (solid line)  and $y^2z-x^2(x+z)$  (dotted line) near $(0,1,0)$, an intersection point of multiplicity $5$. (These are the curves $z=x^3$ and $z=x^3/(1-x^2)$.)} \label{F-cubic_xz}
\end{figure}

{\bf Example 1.} Take 
\begin{eqnarray*}
A(x,y,z)&=& y^2z-x^3\\
B(x,y,z)&=& y^2z-x^2(x+z).
\end{eqnarray*}

Thus the equations $A=0$ and $B=0$ are homogenized versions of the cubic curves $y^2=x^3$ and $y^2=x^2(x+1)$, plotted in Figure 1. We see that they intersect at the origin $(0,0,1)$, but it is not immediately clear what the multiplicity of intersection there is. And are there other intersection points?

Applying our (i.e., Euclid's!) algorithm to $A$ and $B$ as polynomials in $x$, we first have
\begin{eqnarray*}
A(x,y,z)&=& B(x,y,z) +x^2z,\\
\end{eqnarray*}
so that $A\cdot B=A\cdot(x^2z)=2(A\cdot x)+A\cdot z$, using Proposition \ref{a.bprops}(c) and then (b). Then $A\cdot x=(y^2z)\cdot x=2(y\cdot x)+z\cdot x=2(0,0,1)+(0,1,0)$, using \ref{a.bprops}(d), while $A\cdot z=(x^3)\cdot z=3(0,1,0)$. Collecting the results together, we have $A\cdot B= 4(0,0,1)+5(0,1,0)$. Thus $A$ and $B$ intersect at  $(0,0,1)$ with multiplicity $4$ (see Figure 1)  and at $(0,1,0)$ with multiplicity $5$ (Figure 2). Since both curves have degree $3$, and $4+5=3\times 3$, we have checked out B\'ezout's Theorem for this example. Note too that in our standard notation for Galois cycles we have $(0,0,1)=\mathcal C_1(x,y)$ and $(0,1,0)=\mathcal C_0(x)$.

{\bf Example 2.}
 Our second example has been cooked up to give an answer requiring  larger Galois cycles, as well as $(1,0,0)$:
take
\begin{eqnarray*}
A(x,y,z)&=& (y-z)x^5+(y^2-yz)x^4+(y^3-y^2z)x^3\\ & & +(-y^2z^2+yz^3)x^2+
(-y^3z^2+y^2z^3)x-y^4z^2+y^3z^3\\
B(x,y,z)&=&(y^2-2z^2)x^2+(y^3-2yz^2)x+y^4-y^2z^2-2z^4.
\end{eqnarray*}
Applying one step of  Euclid's algorithm to $A$ and $B$ as polynomials in $x$,  we get
$$
A=\frac{(y-z)x(x^2-z^2)}{y^2-2z^2}B+z^2(y-z)(z^2x-y^3);
$$
thus clearing the denominator $y^2-2z^2$ gives
$$
(y^2-2z^2)A=(y-z)x(x^2-z^2)B+(y^2-2z^2)z^2(y-z)(z^2x-y^3).
$$
Then application of (\ref{E-3}) gives 
\begin{equation}\label{E-ab}
A\cdot B=R'\cdot B'+A\cdot G,
\end{equation}
where
\begin{eqnarray*}
R'(x,y,z)&=&z^2(y-z)(z^2x-y^3)\\
B'(x,y,z)&=&x^2+xy+y^2+z^2\\
G(y,z)&=&y^2-2z^2,
\end{eqnarray*}
the $H'\cdot B'$ term not appearing as $H'=1$ here.

Repeating the process with $B'$ and $R'$,  applying (\ref{E-3}) again, and then using Proposition \ref{a.bprops}(b) and (c), we get
\begin{align*}
R'\cdot B'&= (x^2+xy+y^2+z^2)\cdot(z^2(y-z))+(-y^3+xz^2)\cdot ((y^2+z^2)(y^4+z^4))\\
&\qquad -z^4\cdot(-y^3+xz^2)\\
&=2((x^2+xy+y^2)\cdot z)+(x^2+xy+y^2+z^2)\cdot (y-z)\\
&+(-y^3+xz^2)\cdot (y^2+z^2)+(-y^3+xz^2)\cdot(y^4+z^4)
-12(z\cdot y)\\
&=2\sum_{\alpha:\alpha^2+\alpha+1=0}(x-\alpha y)\cdot z+\sum_{\gamma:\gamma^2+\gamma+2=0}(x-\gamma y)\cdot(y-z)\\
&\qquad +\sum_{\beta:\beta^2+1=0}(-y^3+xz^2)\cdot (y-\beta z)\\
&\qquad \qquad +\sum_{\beta:\beta^4+1=0}(-y^3+xz^2)\cdot (y-\beta z)-12(1,0,0)\\
&=2\sum_{\alpha:\alpha^2+\alpha+1=0}(\alpha,1,0)+\sum_{\gamma:\gamma^2+\gamma+2=0}(\gamma,1,1)\\ 
&\qquad +\sum_{\beta:\beta^2+1=0}(-(\beta z)^3+xz^2)\cdot (y-\beta z)\\
&\qquad \qquad +\sum_{\beta:\beta^4+1=0}(-(\beta z)^3+xz^2)\cdot (y-\beta z)-12(1,0,0).\\
\end{align*}
Now 
$$
2\sum_{\alpha:\alpha^2+\alpha+1=0}(\alpha,1,0)+\sum_{\gamma:\gamma^2+\gamma+2=0}(\gamma,1,1)=2\mathcal C_0(x^2+x+1)+\mathcal C_1(x^2+x+2,y-1),
$$
while we can readily show that
$$
\sum_{\beta:\beta^2+1=0}(-(\beta z)^3+xz^2)\cdot (y-\beta z)=4(1,0,0)+\mathcal C_1(x+y,y^2+1),
$$
and
$$
\sum_{\beta:\beta^4+1=0}(-(\beta z)^3+xz^2)\cdot (y-\beta z)=8(1,0,0)+\mathcal C_1(x-y^3,y^4+1).
$$
Thus
$$
R'\cdot B'=2\mathcal C_0(x^2+x+1)+\mathcal C_1(x^2+x+2,y-1)+\mathcal C_1(x+y,y^2+1)
+\mathcal C_1(x-y^3,y^4+1).
$$
So to compute $A\cdot B$ it remains only to evaluate $A\cdot G$. Now
\begin{align*}
A(x,y,z)\cdot G&=A(x,y,z)\cdot(y^2-2z^2)\\
&=\sum_{\beta:\beta^2-2=0}A(x,\beta z,z)\cdot (y-\beta z),
\end{align*}
which we can show equals
$$
\mathcal C_1(x^3-y,y^2-2)+\mathcal C_1(x^2+yx+2,y^2-2)+2(1,0,0).
$$
Hence we obtain  from (\ref{E-ab}) that $A\cdot B$ can be written as a sum of Galois cycles as
\begin{align*}
A\cdot B=&2(1,0,0)+2\mathcal C_0(x^2+x+1)+\mathcal C_1(x^2+x+2,y-1)+\mathcal C_1(x+y,y^2+1)\\
&+\mathcal C_1(x-y^3,y^4+1)+
\mathcal C_1(x^3-y,y^2-2)+\mathcal C_1(x^2+yx+2,y^2-2).
\end{align*}

Once this final form has been obtained, the Galois cycles can be unpacked to write them explicitly as  sums of  points. For instance, $\mathcal C_0(x^2+x+1)=(\omega,1,0)+(\omega^2,1,0)$ where $\omega=\frac{-1+\sqrt{-3}}{2}$, and $\mathcal C_1(x^3-y,y^2-2)=(\gamma,\gamma^3,1)+
(\omega\gamma,\gamma^3,1)+(\omega^2\gamma,\gamma^3,1)+(-\gamma,-\gamma^3,1)+
(-\omega\gamma,-\gamma^3,1)+(-\omega^2\gamma,-\gamma^3,1)$, where $\gamma=2^{1/6}$.

\

 The details of these examples have been given for illustrative purposes only. Of course the algorithm, being deterministic and recursive, is readily automated.

\end{subsection}

\end{section}

\begin{section}{Proof of B{\'e}zout's Theorem}\label{S-B}

We now show that the algorithm described in Section \ref{algorithm}
can be used to give a simple proof of B{\'e}zout's Theorem (Theorem \ref{Th-B}).

\begin{proof} We need to show that $\#(A\cdot B) = \sum_{\mathbf{P}}i_\mathbf{P}(A,B)=mn$.  We proceed by
induction on the $x$-degree of $B$. First suppose that $B$ has $x$-degree
$0$. Then $B$ factors over $\overline K$ into a product of $n$ lines $L$,
so that, by Proposition \ref{a.bprops}(b), $A\cdot B$ is a sum of $n$
intersection cycles $A\cdot L$. From Section \ref{S-lines}, each
$A\cdot L$ is equal to $A'\cdot L$, where $A'$ is a polynomial in two variables of degree $m$, and thus a product of $m$ lines. Hence $A\cdot L$ can be written as a sum of $m$ intersections $L'\cdot L$,
giving $mn$ such intersections in total. Since, by Proposition
\ref{a.bprops}(d), $L'\cdot L$ consists of a single point, we have
$\#(A\cdot B) = mn$ in this case.

Suppose now that $B$ has $x$-degree $k>0$ and that we know that  the
result holds  for all  $B$ with $\xideg{x}B<k$ and for all $A$.
 Then, in the notation of Section \ref{algorithm} we have, by (\ref{E-3}),
\begin{eqnarray*}
\#(A\cdot B)&=&\#(R'\cdot B')- \#(H'\cdot B')+\#(A\cdot G)\\
&=&(\degree R'-\degree H')\degree B'+\degree A\degree G,
\end{eqnarray*}
recalling that $\xideg{x}R'<\xideg{x}B=k$ and
$\xideg{x}H'=\xideg{x}G=0$.

Using the fact that all polynomials involved are homogeneous, we
have from (\ref{Euclid}) that $\degree R'-\degree H'=\degree A$.
Finally, since $\degree B'+\degree G=\degree B$ from $B=B'G$, the result
$\#(A\cdot B)=\degree A\,\degree B=mn$ follows for $\xideg{x}B=k$. This
proves the inductive step.

\end{proof}

\end{section}

\begin{section}{Appendix: Intersection multiplicity \\ of algebraic curves}\label{prelims}

In Section \ref{S-three}, we used the properties of intersection cycles $A\cdot B$ given in Proposition \ref{a.bprops} without actually defining  intersection multiplicity $i_\mathbf{P}(A,B)$. In order to make this article completely self-contained, we now give this definition, and derive the properties that we need to prove Proposition \ref{a.bprops}. This is standard material, which can be found, for instance, in \cite{F} or \cite{K}.

Let $A,B\in K[x,y,z]$ be algebraic curves with $\gcd(A,B)=1$.  Define the {\it local ring of rational functions of degree $0$} at $\mathbf{P}\in\overline  K{\mathbb P}^2$ to be
\begin{eqnarray*}
R_\mathbf{P}&=&\left\{\frac{S}{T}: S,T\in \overline K[x,y,z], \degree S=\degree T, T(\mathbf{P})\neq0\right\},
\end{eqnarray*}
where all polynomials are homogeneous. Further, define
\begin{eqnarray*}
(A,B)_\mathbf{P}&=&\left\{\frac{S}{T}\in R_\mathbf{P}: S=MA+NB, M,N,T\in \overline K[x,y,z], T(\mathbf{P})\neq 0\right\},
\end{eqnarray*}
 the ideal generated by $A$ and $B$ in $R_\mathbf{P}$.

Following \cite{F}, we can now define the intersection multiplicity
$i_\mathbf{P}(A,B)$ of $A$ and $B$ to be the dimension of the $\overline K$-vector space
 $R_\mathbf{P}/(A,B)_\mathbf{P}$ (and so equal to $0$ if $(A,B)_\mathbf{P}=R_\mathbf{P}$).

\begin{lemma}\label{L-i}
Let $\mathbf{P}\in\overline K{\mathbb P}^2$ and $A,B,C\in K[x,y,z]$ with $\gcd(A,B)=\gcd(A,C)=1$. Then
\begin{itemize}
\item[$(a)$] $i_\mathbf{P}(A,B)>0$ if and only if $\mathbf{P}$  lies on both $A$ and $B$;
\item[$(b)$] $i_\mathbf{P}(A,B)=i_\mathbf{P}(B,A)$;
\item[$(c)$] $i_\mathbf{P}(A,BC)=i_\mathbf{P}(A,B)+i_\mathbf{P}(A,C)$;
\item[$(d)$] $i_\mathbf{P}(A,B+AC)=i_\mathbf{P}(A,B)$ if $\degree (AC)=\degree  B$;
\item[$(e)$] For distinct lines $L,L'$, the only point on both lines is $\mathbf{P}_\times$ given by (\ref{l1.l2}), and $i_{\mathbf{P}_\times}(L,L')=1$.
\end{itemize}
\end{lemma}

\begin{proof}

To prove (a), take $S/T\in R_\mathbf{P}$. If $\mathbf{P}$ is {\it not} on both $A$ and $B$, then
$S/T=AS/AT=BS/BT\in(A,B)_\mathbf{P}$, since at least one of $AT$ and $BT$ is nonzero at $\mathbf{P}$. Hence $R_\mathbf{P}=(A,B)_\mathbf{P}$, so that $i_\mathbf{P}(A,B)=0$. On the other hand, if $\mathbf{P}$ {\it is} on both $A$ and $B$, then all elements of $(A,B)_\mathbf{P}$ are $0$ at $\mathbf{P}$, while the constant $1=1/1$ clearly is not! Hence $R_\mathbf{P}/(A,B)_\mathbf{P}$ is at least one-dimensional.

Properties (b) and (d) are immediately obvious, since $(A,B)_\mathbf{P}=(B,A)_\mathbf{P}$ and $(A,B+AC)_\mathbf{P}=(A,B)_\mathbf{P}$.

For (c), we base our argument on that in \cite[p. 77]{F}. Define two maps
\begin{eqnarray*}
\psi&:&\frac{R_\mathbf{P}}{(A,C)_\mathbf{P}}\rightarrow\frac{R_\mathbf{P}}{(A,BC)_\mathbf{P}},~\overline{w}\mapsto\overline{bw}\\
\phi&:&\frac{R_\mathbf{P}}{(A,BC)_\mathbf{P}}\rightarrow\frac{R_\mathbf{P}}{(A,B)_\mathbf{P}},~\overline{w}\mapsto\overline{w},
\end{eqnarray*}
where $\overline{w}$ denotes the residue of $w\in R_\mathbf{P}$ in the corresponding quotient ring, and $b=B/V^{n}$, where $n=\degree B$ and $V$ is one of $x$, $y$, or $z$,
chosen so that it is nonzero at $\mathbf{P}$.

It is easy to check that both maps $\phi$ and $\psi$ are $\overline K$-linear maps.
We claim that the sequence
\begin{eqnarray*}
\begin{CD}
0@>>>\displaystyle\frac{R_\mathbf{P}}{(A,C)_\mathbf{P}}@>\psi>>\displaystyle\frac{R_\mathbf{P}}{(A,BC)_\mathbf{P}}@>\phi>>\displaystyle\frac{R_\mathbf{P}}{(A,B)_\mathbf{P}}@>>>0
\end{CD}
\end{eqnarray*}
is exact.

Supposing that $\overline{w}\in\ker\psi$, we get $bw\in(A,BC)_\mathbf{P}$ which, on multiplying by $V^nU$, say, to clear denominators, gives $SA=B(D-TC)$ for some $D,S,T\in \overline K[x,y,z]$, with $w=D/U$. As $A$ and $B$ have no common factor, $A$ must divide $D-TC$, so that, on dividing by $U$, we have $w=D/U\in(A,C)_\mathbf{P}$, Hence $\overline{w}=0$, and $\psi$ is injective.

It is easy to show that $\text{im}\psi=\ker\phi$, by checking  inclusion in both directions. Also, it is clear that $\phi$ is surjective, completing the verification of exactness.
By the rank-nullity theorem from linear algebra, this then implies (c).

To prove (e), take $A$ and $B$ to be the lines of Proposition \ref{a.bprops}(d). We first note that, by Cramer's rule, the point $\mathbf{P}_\times$ is the (only) point common to both lines, so that, by  Lemma \ref{L-i}(a), $A\cdot B$ is a positive integer multiple of $\mathbf{P}_\times$. We need to show that this multiple is indeed $1$.

Take a third line $C=c_1x+c_2y+c_3z$ so that the matrix
$$
J=\left(\begin{matrix} a_1 & a_2 & a_3\\b_1 & b_2 & b_3\\c_1 & c_2 & c_3\end{matrix}\right)
$$
has nonzero determinant. (This is always possible, as $K^3$ is $3$-dimensional!) Then 
$$
J^{-1}\left(\begin{matrix} A\\B\\C\end{matrix}\right) =\left(\begin{matrix} x\\y\\z\end{matrix}\right),
$$
so that any polynomial in $\overline K[x,y,z]$ can be written as a polynomial in $\overline K[A,B,C]$.
Thus any element $q$ of $R_{\mathbf{P}_\times}$ can be written in the form
$$
q=\frac{AS_1(A,B,C)+BS_2(B,C)+s_0C^k}{AT_1(A,B,C)+BT_2(B,C)+t_0C^k}
$$
for $s_0,t_0\in \overline K$ with $t_0\ne 0$, some positive integer $k$, and polynomials $S_1,S_2,T_1$, and $T_2$. Then, by putting the difference $q-s_0/t_0$ over a common denominator, we see that it belongs to $(A,B)_{\mathbf{P}_\times}$. Hence ${R_{\mathbf{P}_\times}}/{(A,B)_{\mathbf{P}_\times}}$ is spanned by $1$, and so is one-dimensional; thus $i_{\mathbf{P}_\times}(A,B)=1$.
\end{proof}

\end{section}

{\bf ACKNOWLEDGEMENTS.} We are pleased to thank the referees, and Liam O'Carroll, for their constructive comments and suggestions.

\noindent {\bf Jan Hilmar} received his B.A. from St. Mary's College of Maryland in 2004, and his Ph.D. at the University of Edinburgh in 2008. When he is not on his bike, he is doing his National Service in a refugee home in his hometown of Vienna, and working as a freelance web developer.

\noindent{\it trafficjan82@gmail.com}

\

\noindent{\bf Chris Smyth} received his B.A. from the Australian National University in 1968, and his Ph.D. in number theory from the University of Cambridge in 1972. After spells in Finland, England, Australia and Canada, he was, when the music stopped,  happy to find himself in Edinburgh, Scotland. He likes walking, sometimes accompanied by Mirabelle, his cat.

\noindent {\it School of Mathematics, University of Edinburgh, Mayfield Road, Edinburgh EH9 3JZ, UK.

\noindent c.smyth@ed.ac.uk}
\end{document}